\documentclass[reqno,twoside,12pt]{amsart}

\hoffset - 1.5cm

\usepackage{amsmath}
\usepackage{amssymb}
\usepackage{graphicx}
\usepackage{amstext}
\usepackage{amsopn}
\usepackage{amsfonts}
\usepackage{xcolor}
\usepackage{setspace}
\usepackage{enumerate}
\usepackage[polish,english]{babel}
\usepackage[utf8]{inputenc}
\usepackage{polski}

\usepackage{lmodern}

\newtheorem{Remark}{Remark}[section]
\newtheorem{Corollary}[Remark]{Corollary}
\newtheorem{Definition}[Remark]{Definition}
\newtheorem{Example}[Remark]{Example}
\newtheorem{Fact}[Remark]{Fact}
\newtheorem{Lemma}[Remark]{Lemma}
\newtheorem{Proposition}[Remark]{Proposition}
\newtheorem{Theorem}[Remark]{Theorem}

\newcommand{\ba}{\begin{array}}
\newcommand{\bc}{\begin{center}}
\newcommand{\bd}{\begin{description}}
\newcommand{\bdm}{\begin{displaymath}}
\newcommand{\be}{\begin{enumerate}}
\newcommand{\beq}{\begin{equation}}
\newcommand{\bdf}{\begin{Definition}}
\newcommand{\bex}{\begin{Example}}
\newcommand{\bft}{\begin{Fact}}
\newcommand{\bl}{\begin{Lemma}}
\newcommand{\bp}{\begin{Proposition}}
\newcommand{\br}{\begin{Remark}}
\newcommand{\bt}{\begin{Theorem}}
\newcommand{\bco}{\begin{Corollary}}
\newcommand{\bhy}{\begin{Hypothesis}}
\newcommand{\ea}{\end{array}}
\newcommand{\ec}{\end{center}}
\newcommand{\ed}{\end{description}}
\newcommand{\edm}{\end{displaymath}}
\newcommand{\ee}{\end{enumerate}}
\newcommand{\eeq}{\end{equation}}
\newcommand{\edf}{\end{Definition}}
\newcommand{\eex}{\end{Example}}
\newcommand{\eft}{\end{Fact}}
\newcommand{\el}{\end{Lemma}}
\newcommand{\ep}{\end{Proposition}}
\newcommand{\er}{\end{Remark}}
\newcommand{\et}{\end{Theorem}}
\newcommand{\eco}{\end{Corollary}}
\newcommand{\ehy}{\end{Hypothesis}}

\newcommand{\bH}{\mathbb{H}}

\newcommand{\bN}{\mathbb{N}}

\newcommand{\bR}{\mathbb{R}}

\newcommand{\bV}{\mathbb{V}}
\newcommand{\bW}{\mathbb{W}}
\newcommand{\bX}{\mathbb{X}}
\newcommand{\bY}{\mathbb{Y}}
\newcommand{\bZ}{\mathbb{Z}}
\newcommand{\cA}{\mathcal{A}}
\newcommand{\cB}{\mathcal{B}}
\newcommand{\cC}{\mathcal{C}}
\newcommand{\cD}{\mathcal{D}}

\newcommand{\cL}{\mathcal{L}}

\newcommand{\cN}{\mathcal{N}}

\newcommand{\cR}{\mathcal{R}}
\newcommand{\cS}{\mathcal{S}}

\newcommand{\cX}{\mathcal{X}}
\newcommand{\cY}{\mathcal{Y}}
\newcommand{\cZ}{\mathcal{Z}}

\newcommand{\im}{\mathrm{ im \;}}

\numberwithin{equation}{section} \errorcontextlines=0

\newcommand{\cl}{\mathrm{cl}}

\newcommand{\sone}{S^1}

\newcommand{\ds}{\displaystyle}

\newcommand{\subga}{\overline{\mathrm{sub}}(\Gamma)}

\newcommand{\sub}{\overline{\mathrm{sub}}}

\newcommand{\cw}{\mathrm{CW}}
\newcommand{\gcw}{\Gamma\text{-CW}}
\newcommand{\hcw}{H\text{-CW}}

\newcommand{\h}{\mathbb{H}}

\makeatletter
\@namedef{subjclassname@2020}{%
  \textup{2020} Mathematics Subject Classification}
\makeatother

\usepackage{todonotes}
\usepackage{setspace}

\begin{document}

\title[Symmetric Newtonian systems]{Periodic solutions to symmetric Newtonian systems in neighborhoods of orbits of equilibria}

\author{Anna Go{\l}\c{e}biewska}
\address{Faculty of Mathematics and Computer Science \\ Nicolaus Copernicus University in Toru\'n\\
PL-87-100 Toru\'{n} \\ ul. Chopina $12 \slash 18$ \\
Poland}

\author{Marta Kowalczyk}

\author{S{\l}awomir Rybicki}

\author{Piotr Stefaniak}

\email{aniar@mat.umk.pl (A. Go{\l}\c{e}biewska)}
\email{martusia@mat.umk.pl (M. Kowalczyk)}
\email{rybicki@mat.umk.pl (S. Rybicki)}
\email{cstefan@mat.umk.pl (P. Stefaniak)}

\date{\today}


\begin{abstract}
The aim of this paper is to prove the existence of periodic solutions to symmetric Newtonian systems in any neighborhood of an isolated orbit of equilibria. Applying equivariant bifurcation techniques we obtain a generalization of the classical Lyapunov center theorem to the case of symmetric potentials with orbits of non-isolated critical points. Our tool is an equivariant version of the Conley index. To compare the indices we compute cohomological dimensions of some orbit spaces.
\end{abstract}

\maketitle

 \bc  {\bf On  the occasion of the 75th birthday of Professor E. N. Dancer}\ec

\section{Introduction}

One of the most important classical problems of differential equations is the search for non-stationary periodic solutions to Hamiltonian and Newtonian systems in a neighborhood of an isolated equilibrium. This problem has a long history and has been studied by many mathematicians for centuries, see the following classical articles due to Lyapunov \cite{lyapunov}, Weinstein \cite{weinstein},  Moser \cite{moser},  Fadell and Rabinowitz \cite{fadrab},  Montaldi, Roberts and Stewart \cite{morost}, Bartsch \cite{bartsch1} and references therein. We are aware that this list is far from being complete. Note that the stationary solutions considered in these papers have been assumed to be non-degenerate. It is worth to point out that non-degenerate as well as isolated degenerate   stationary solutions have been considered by  Dancer and the third author in \cite{danryb} and Szulkin in \cite{szulkin}.

In this article we focus our attention on symmetric Newtonian systems.
 More precisely, we consider $\bR^n$ as an orthogonal representation of a compact Lie group $\Gamma$ and a $\Gamma$-invariant potential $U\colon \bR^n \to \bR$ of class $C^2$, i.e. the potential $U$ satisfying $U(\gamma u)=U(u)$ for all $\gamma \in \Gamma$ and $u \in \bR^n.$ If $u_0$ is a critical point of $U,$ i.e. $\nabla U(u_0)=0,$ then the orbit $\Gamma(u_0) =  \{\gamma u_0\colon \gamma \in \Gamma\}$ consists of critical points of $U$, i.e.
$\Gamma (u_0) \subset (\nabla U)^{-1}(0).$  It is known that the orbit $\Gamma(u_0)$ is $\Gamma$-homeomorphic to $\Gamma \slash \Gamma_{u_0},$ where $\Gamma_{u_0}=\{\gamma \in \Gamma\colon \gamma u_0=u_0\}$ is the stabilizer of $u_0.$ Hence if $\dim \Gamma \geq 1,$ then  it can happen that $\dim \Gamma (u_0) \geq 1$, i.e. the critical point $u_0$ is not  isolated  in $(\nabla U)^{-1}(0).$

We study the existence of non-stationary periodic solutions of the following $\Gamma$-symmetric system
\beq \label{equ} \ddot u(t) = - \nabla U(u(t)) \eeq in any neighborhood of the isolated orbit $\Gamma (u_0) \subset (\nabla U)^{-1}(0).$
In other words, we are going to prove a symmetric version of the classical Lyapunov center theorem, where, due to additional $\Gamma$-symmetries, an isolated stationary solution is replaced by an isolated orbit of stationary solutions. 
These theorems, which are the main results of our paper, are formulated below.

Define $\sigma^+(\nabla^2 U(u_0))=\sigma(\nabla^2 U(u_0)) \cap (0,+\infty)$  and put $\cB=\{\beta_1,\beta_2,\ldots,\beta_q\}$, where $\beta_1>\beta_2>\ldots>\beta_q>0$ are such that  $\sigma^+(\nabla^2 U(u_0)) = \{\beta_1^2,\ldots, \beta_q^2\}$.

\bt \label{main-theo1} [Symmetric Lyapunov center theorem for a non-degenerate orbit]
Let  $\Omega \subset \bR^n$ be an open and $\Gamma$-invariant subset of an orthogonal representation $\bR^n$ of a compact Lie group $\Gamma$.
Assume that $U\colon\Omega \to \bR$ is a $\Gamma$-invariant potential of class $C^2$ and $u_0 \in \Omega \cap (\nabla U)^{-1}(0)$.
If
\be
\item $\Gamma_{u_0}=S^1$ or $\Gamma_{u_0}=\bZ_m$ for some $m \in \bN,$
\item $\dim\ker\nabla^2 U(u_0)=\dim \Gamma(u_0),$
\item $\sigma^+(\nabla^2 U(u_0))\neq \emptyset$,
\ee
then, for any $\beta_{j_0}\in\cB$ such that $\beta_j \slash \beta_{j_0} \not \in \bN $ for $\beta_j\in\cB\setminus\{\beta_{j_0}\}$, there exists a sequence $(u_k)$ of periodic solutions of the system \eqref{equ} with a sequence $(T_k)$ of minimal periods such that $T_k\to 2 \pi \slash \beta_{j_0}$ and   for any $\varepsilon>0$ there exists $k_0\in\bN$ such that $u_k([0,T_k])\subset \Gamma(u_0)_\varepsilon=\bigcup_{u\in \Gamma(u_0)} B_\epsilon(\bR^n,u)$ for all $k\geq k_0.$
\et

\bt\label{main-theo2} [Symmetric Lyapunov center theorem for a minimal orbit]
Let  $\Omega \subset \bR^n$ be an open and $\Gamma$-invariant subset of an orthogonal representation $\bR^n$ of a compact Lie group $\Gamma$.
Assume that $U\colon\Omega \to \bR$ is a $\Gamma$-invariant potential of class $C^2$ and $u_0 \in \Omega \cap (\nabla U)^{-1}(0)$. If moreover,
\be
\item $\Gamma_{u_0}=S^1$ or $\Gamma_{u_0}=\bZ_m$ for some $m \in \bN,$
\item $\Gamma(u_0)$ consists of minima of the potential $U$,
\item $\Gamma(u_0)$ is isolated in $(\nabla U)^{-1}(0)$,
\item $\sigma^+(\nabla^2 U(u_0))\neq \emptyset$,
\ee
then, for any $\beta_{j_0}\in\cB$ such that $\beta_j \slash \beta_{j_0} \not \in \bN $ for $\beta_j\in\cB\setminus\{\beta_{j_0}\}$, there exists a sequence $(u_k)$ of periodic solutions of the system \eqref{equ} with a sequence $(T_k)$ of minimal periods such that $T_k\to 2 \pi \slash \beta_{j_0}$ and  for any $\varepsilon>0$ there exists $k_0\in\bN$ such that $u_k([0,T_k])\subset \Gamma(u_0)_\varepsilon$ for all $k\geq k_0.$
\et

Note that if the assumptions of the above theorems are satisfied, then obviously $\beta_{1}\in \cB$ and $\beta_j \slash \beta_{1} \not \in \bN $ for $\beta_j\in\cB\setminus\{\beta_{1}\}$, which implies that there exists at least one sequence $(u_k)$ of periodic solutions of the system \eqref{equ} with a sequence $(T_k)$ of minimal periods converging to $2 \pi \slash \beta_{1}$. Moreover, selecting various $\beta_{j_0}$ satisfying the assumptions we obtain different sequences of periodic solutions which can be distinguished by their minimal periods. 


We emphasize that in Theorem \ref{main-theo1} we consider a non-degenerate orbit $\Gamma(u_0)$, i.e. an orbit satisfying the condition  $\dim \ker \nabla^2 U(u_0) = \dim \Gamma (u_0)$, whereas in Theorem \ref{main-theo2} we allow the orbit $\Gamma (u_0)$ to be degenerate, i.e. such that  $\dim \ker \nabla^2 U(u_0) > \dim \Gamma (u_0).$

Results of this type have been proved in  \cite{PRS,PRS2} for symmetric Newtonian systems and in \cite{Strzelecki} for symmetric Hamiltonian systems under the assumption that the stabilizer $\Gamma_{u_0}$ is trivial, i.e. the orbit $\Gamma (u_0)$ is $\Gamma$-homeomorphic to the group $\Gamma.$
On the other hand, in \cite{KPR} we have considered symmetric Newtonian systems under the assumption that the stabilizer $\Gamma_{u_0}$ is isomorphic to a finite-dimensional torus  $T \subset \Gamma$, i.e. the orbit $\Gamma (u_0)$ is $\Gamma$-homeomorphic to $\Gamma \slash T.$

In our paper we consider orbits $\Gamma (u_0)$ such that the stabilizer $\Gamma_ {u_0}$ equals $\sone $ or $\bZ_m,$ i.e. the orbit $\Gamma(u_0)$ is $\Gamma$-homeomorphic to $\Gamma \slash \sone$ or $\Gamma \slash \bZ_m.$ It is worth pointing out that the case $\Gamma_{u_0} = \bZ_m$ has not been considered in our previous articles.

In \cite{KPR, PRS,PRS2,Strzelecki} the solutions of the system \eqref{equ} have been considered as orbits of critical points of a family of $(\Gamma \times \sone)$-invariant functionals defined on an appropriately chosen Hilbert space, which is an orthogonal  representation of the group $\Gamma \times \sone$. To prove the main results of these articles we have applied the techniques of equivariant bifurcation theory.  As a topological tool we have used the infinite-dimensional generalization of the $(\Gamma \times \sone)$-equivariant Conley index defined by Izydorek, see \cite{Izydorek}. To show the existence of non-stationary periodic solutions of the  system \eqref{equ} in an arbitrary neighborhood of the orbit $\Gamma (u_0) $ we have proved the change of this index.
To distinguish between two Conley indices we have used the equivariant Euler characteristic, which is an element of the Euler ring $U(\Gamma \times \sone),$ see \cite{dieck}. These calculations were rather tedious and complicated.

In this paper we apply a different approach. Namely, we consider orbit spaces of the equivariant Conley indices, which are the homotopy types of the weighted projective spaces or the lens complexes. Since the cohomology groups of the weighted projective spaces and the lens complexes are known, see \cite{kawasaki}, we compute the cohomological dimensions of these spaces. This technique of distinguishing between equivariant Conley indices seems to be much simpler.

To be more precise, instead of making calculations in the Euler ring $U(\Gamma \times \sone)$ we employ the concept of a quotient space and use the notion of the cohomological dimension to prove a change of the equivariant Conley index.

The topological results described above are given in Section \ref{sec:abstract}, whereas Section \ref{sec:proofs} contains the proofs of Theorems \ref{main-theo1} and \ref{main-theo2}. Additionally, in the appendix we recall some relevant material needed in the previous sections.

\section{Abstract results}\label{sec:abstract}

In this section we study the $\Gamma$-homotopy equivalence of some finite pointed $\gcw$-complexes (see \cite{dieck} for the definition), where $\Gamma$ is a compact Lie group.
In particular, we are interested in complexes being of the form of smash products over some groups
$H \in \subga,$ where $\subga$ denotes the set of closed subgroups of $\Gamma$.
Below we recall the definition of such a space, see \cite{dieck} for details.

Fix $H \in \subga$ and let $\bX$ be a pointed $H$-space with a base point $\ast.$ Denote by  $\Gamma^+$  the group $\Gamma$ with a disjoint $\Gamma$-fixed base point added.
Recall that the smash product of $\Gamma^+$ and $\bX$ is $\Gamma^+ \wedge \bX=\Gamma^+ \times \bX \slash \Gamma^+ \vee \bX =\Gamma \times \bX \slash \Gamma \times \{\ast\}.$
The group $H$ acts on the pointed space $\Gamma^+ \wedge \bX$ by $(h,[\gamma,y]) \mapsto [\gamma h^{-1},hy].$
We denote the orbit space of this action by $\Gamma^+ \wedge_H \bX$ and call it the smash over $H.$ The formula
$(\gamma',[\gamma,y]) \mapsto [\gamma'\gamma,y]$
induces a $\Gamma$-action so that $\Gamma^+ \wedge_H \bX$ becomes a pointed $\Gamma$-space.

\begin{Remark}
If $\bX$ is a finite pointed $H$-CW-complex, then  $\Gamma^+ \wedge_H \bX$ is a $\Gamma$-CW complex, see \cite{PRS}, and the orbit space $\bX \slash H$ is a finite pointed $\cw$-complex, see \cite{dieck}.
\end{Remark}

For brevity we will write $\bX \approx_H \bY$ if $H$-CW-complexes $\bX$, $\bY$ are $H$-homotopically equivalent. In the non-equivariant case, $\cX \approx \cY$ denotes homotopical equivalence of CW-complexes $\cX$, $\cY$.

\begin{Lemma}\label{lem:Gsmashtosmash}
Let  $\bX,\bY$ be finite pointed $\hcw$-complexes. If $\Gamma^+ \wedge_H \bX \approx_\Gamma  \Gamma^+ \wedge_H \bY$, then $\bX \slash H  \approx \bY \slash H.$
\end{Lemma}

\begin{proof}
Consider the orbit spaces $(\Gamma^+ \wedge_H \bX) \slash \Gamma$, $(\Gamma^+ \wedge_H \bY) \slash \Gamma$ and note that from the assumption and the formula (1.12) of \cite[Chapter 1]{dieck} we get
\begin{equation}\label{eq:Gsmashtosmash1}
(\Gamma^+ \wedge_H \bX) \slash \Gamma \approx (\Gamma^+ \wedge_H \bY) \slash \Gamma.
\end{equation}
For an $H$-space $\bW$ denote by $\Gamma\times_H\bW$ the twisted product over $H$. The inclusion $\bW \to \Gamma \times_H \bW$ induces a homeomorphism of $\bW \slash H$ and $(\Gamma \times_H \bW) \slash \Gamma$, see for instance \cite{dieck}, \cite{Kawakubo}. Moreover, if $\bW$ is a pointed $H$-space, $\Gamma^+\wedge_H\bW  = (\Gamma\times_H\bW)/(\Gamma\times_H\{\ast\})$.
Therefore,
\begin{equation}\label{eq:Gsmashtosmash2}
(\Gamma^+ \wedge_H \bX) \slash \Gamma \approx \bX \slash H\ \text{ and }\ (\Gamma^+ \wedge_H \bY) \slash \Gamma \approx \bY \slash H.
\end{equation}
Combining \eqref{eq:Gsmashtosmash1} and \eqref{eq:Gsmashtosmash2} we obtain the assertion.
\end{proof}

From this lemma it follows that in some cases we can reduce comparing the $\Gamma$-equivariant homotopy types of $\Gamma$-CW-complexes of the form $\Gamma^+ \wedge_H \bX$ to comparing the homotopy types of CW-complexes of the form $\bX \slash H$.

To consider the homotopy equivalence of CW-complexes we will use the notion of the cohomological dimension of a $\cw$-complex.
For a $\cw$-complex $\cZ$ denote by $\widetilde H^k(\cZ;\bZ)$ the $k$-th reduced cohomology group of $\cZ,$ see \cite{hatcher}.
If there exists a number $k\geq 0$ such that $\widetilde H^k(\cZ;\bZ) \not = 0$, then we define the cohomological dimension of a $\cw$-complex $\cZ$ by
$$\cC\cD(\cZ):=\max\{k \in \bN \cup \{0\}\colon \widetilde H^k(\cZ;\bZ) \not = 0\}.$$
Obviously, if $\cZ_1, \cZ_2$ are homotopically equivalent $\cw$-complexes, then $\cC\cD(\cZ_1)=\cC\cD(\cZ_2)$.

\begin{Lemma}\label{lem:smash} Let $\cZ_1, \cZ_2$ be finite $\cw$-complexes such that $\cC\cD(\cZ_1)$ and $\cC\cD(\cZ_2)$ are well-defined. Then  we have
$\cC\cD(\cZ_1 \wedge \cZ_2) = \cC\cD(\cZ_1) + \cC\cD(\cZ_2).$
\end{Lemma}
\begin{proof}
Since $\cZ_1, \cZ_2$ are finite CW-complexes, the groups $\widetilde H^{q_1}(\cZ_1;\bZ)$, $\widetilde H^{q_1}(\cZ_2;\bZ)$ are finitely generated free groups for every $q_1,q_2\in\bN\cup\{0\}$. Therefore,
by the K$\ddot{\rm u}$nneth formula, see \cite{hatcher}, we obtain the reduced cross product isomorphism
$$\times\colon \widetilde H^{q_1}(\cZ_1;\bZ) \otimes  \widetilde H^{q_2}(\cZ_2;\bZ) \to \widetilde H^{q_1+q_2}(\cZ_1 \wedge \cZ_2;\bZ).$$
That is why  the reduced cross product $$\times\colon \widetilde H^{\cC\cD(\cZ_1)}(\cZ_1;\bZ) \otimes \widetilde H^{\cC\cD(\cZ_2)}(\cZ_2;\bZ) \to \widetilde H^{\cC\cD(\cZ_1) + \cC\cD(\cZ_2)}(\cZ_1 \wedge \cZ_2;\bZ)$$
is an isomorphism of non-trivial groups.

What is left is to show that $\widetilde H^{q}(\cZ_1 \wedge \cZ_2;\bZ)=0$ for any $q >\cC\cD(\cZ_1) + \cC\cD(\cZ_2). $ Note that if  $q > \cC\cD(\cZ_1) + \cC\cD(\cZ_2) $ and $q=q_1+q_2$ then
 $q_1  > \cC\cD(\cZ_1) \text{ or } q_2 > \cC\cD(\cZ_2).$
Therefore the reduced cross product
$$\times\colon \widetilde H^{q_1}(\cZ_1;\bZ) \otimes  \widetilde H^{q_2}(\cZ_2;\bZ) \to \widetilde H^{q_1+q_2}(\cZ_1 \wedge \cZ_2;\bZ)$$ is an isomorphism of trivial groups, which completes the proof.
\end{proof}

From now on, we  assume that $H$ is a closed subgroup of $S^1$, i.e. $H\in\sub(S^1)=\{S^1,\bZ_1,\bZ_2\ldots \}$.
Let $\bV$ be an orthogonal representation of the group $H$ with the isotypical decomposition
\begin{equation}\label{eq:isotyp}
\bV  =  \bR[k_0,0] \oplus \bR[k_1,m_1] \oplus \ldots \oplus \bR[k_p,m_p]=\bR[k_0,0] \oplus \widehat\bV,
\end{equation}
where $\widehat\bV=\bR[k_1,m_1] \oplus \ldots \oplus \bR[k_p,m_p]$, $k_0 \in \bN \cup \{0\}$, $k_1,\ldots, k_p \in \bN$ and $m_1,\ldots,m_p\in\bN$ (in the case $H=\bZ_m$ we assume that $m_1,\ldots,m_p \in \{1, \ldots, m-1\}$), see Appendix for more details.

Denote by $S^\bV$ the one-point compactification of $\bV$.
It is known that the space $S^\bV$ is the $H$-homotopy type of a finite pointed $H$-CW-complex, see for instance \cite{Illman}, \cite{mayer}.

\begin{Lemma}\label{lem:CD}
Let $\bV$ be an $H$-representation with the isotypical decomposition given by \eqref{eq:isotyp}. Then
\begin{equation}\label{eq:CD}
\cC\cD(S^\bV \slash H)=
\left\{\begin{array}{lcr}
\displaystyle k_0 +2 \sum_{i=1}^p k_i & \text{ if } & H=\bZ_m, \\
\displaystyle k_0+2 \sum_{i=1}^p k_i -1 & \text{ if } & H=S^1.
\end{array}
\right.
\end{equation}
\end{Lemma}

\begin{proof}
First note that the group $H$ acts trivially on $\bR[k_0,0]$, therefore we can identify $S^{\bR[k_0,0]}$ with $S^{k_0}$. Moreover,
$$S^\bV \slash H =(S^{\widehat\bV\oplus\bR[k_0,0]}) \slash H= (S^{\widehat\bV} \wedge S^{k_0}) \slash H=(S^{\widehat\bV}  \slash H) \wedge S^{k_0}$$
and from Lemma \ref{lem:smash} we obtain
$$\cC\cD(S^\bV \slash H) = \cC\cD((S^{\widehat\bV}  \slash H) \wedge S^{k_0}) = \cC\cD(S^{\widehat\bV}\slash H) + \cC\cD(S^{k_0})=\cC\cD(S^{\widehat\bV}\slash H)+k_0.$$

To find $\cC\cD(S^{\widehat\bV}\slash H)$ we will first show that
\begin{equation}\label{eq:CDort}
\cC\cD(S^{\widehat\bV} \slash H)=\cC\cD(S(\widehat\bV) \slash H)+1,
\end{equation}
where $S(\widehat\bV)$ is the unit sphere in $\widehat\bV$.
To this end consider the unreduced suspension of $S(\widehat\bV)$, denoted by $\Sigma S(\widehat\bV)$, and note that
$S^{\widehat\bV} \approx \Sigma S(\widehat\bV)$. Moreover, $\Sigma (S(\widehat\bV)) \slash H = \Sigma(S(\widehat\bV)\slash H)$.
Therefore, by the suspension isomorphism of cohomology theory, we obtain
$$\widetilde H^{k+1}(S^{\widehat\bV}\slash H;\bZ) =\widetilde  H^{k+1}(\Sigma (S(\widehat\bV)) \slash H;\bZ) =\widetilde H^{k+1}(\Sigma(S(\widehat\bV)\slash H);\bZ) =\widetilde H^k(S(\widehat\bV) \slash H;\bZ) $$
for $k \geq 0$. From this we get \eqref{eq:CDort}.

To compute $\cC\cD(S(\widehat\bV) \slash H)$ we will apply the results of \cite{kawasaki}. Note that $S(\widehat\bV) \slash \bZ_m$ is a lens complex and $S(\widehat\bV) \slash \sone$ is a twisted projective space, see Appendix.
Therefore by Theorem 2 of \cite{kawasaki} we obtain
$$\ds \cC\cD(S(\widehat\bV) \slash \bZ_m)=2 \sum_{i=1}^p k_i -1.$$
 Similarly, applying Theorem 1 of the same paper, we have
$$\ds \cC\cD(S(\widehat\bV) \slash \sone)=2 \sum_{i=1}^p k_i -2.$$
Using \eqref{eq:CDort} we thus obtain \eqref{eq:CD}.
\end{proof}

From Lemma \ref{lem:CD} we immediately obtain the following corollary.
\begin{Corollary}\label{cor:CDvsdim}
If $H\in\sub(S^1)$ and $\bV$ is an orthogonal $H$-representation, then
$$
\cC\cD(S^\bV \slash H)=
\left\{\begin{array}{lcr}
\dim \bV & \text{ if } & H=\bZ_m, \\
\dim\bV-1 & \text{ if } & H=S^1.
\end{array}
\right.
$$
\end{Corollary}

We are now in a position to prove the main theorem of this section.

\begin{Theorem} \label{thm:general}
Suppose that $H\in\sub(S^1)$ and $\bV_1$, $\bV_2$ are orthogonal $H$-representations such that $\dim\bV_1\neq\dim\bV_2$. Then
$\Gamma^+ \wedge_H S^{\bV_1} \not\approx_{\Gamma} \Gamma^+ \wedge_H S^{\bV_2}$.
\end{Theorem}

\begin{proof}
Let  $\bV_1$, $\bV_2$ satisfy the assumptions of the theorem. Suppose that
$\Gamma^+ \wedge_H S^{\bV_1} \approx_{\Gamma} \Gamma^+ \wedge_H S^{\bV_2}$.
Then, from Lemma \ref{lem:Gsmashtosmash}, the CW-complexes $S^{\bV_1}/H$ and $S^{\bV_2}/H$ are homotopically equivalent and therefore $\cC\cD(S^{\bV_1} \slash H)=\cC\cD(S^{\bV_2} \slash H)$.
Hence, by Corollary \ref{cor:CDvsdim}, $\dim\bV_1=\dim\bV_2$, which completes the proof.
\end{proof}

\section{Proofs of Theorems \ref{main-theo1} and \ref{main-theo2}}\label{sec:proofs}

In this section we prove the symmetric Lyapunov center theorems formulated in Introduction, namely Theorems \ref{main-theo1} and \ref{main-theo2}.
We start with reviewing some classical facts on the variational setting for our problem. The material is standard and well known, see for instance \cite{mawhin} for more details.

Assume that $\Omega\subset\mathbb{R}^n$ is an open and $\Gamma$-invariant subset of an orthogonal representation $\mathbb{R}^n$ of a compact Lie group $\Gamma$ and $U\colon\Omega\rightarrow\mathbb{R}$ is a $\Gamma$-invariant potential of class $C^2.$
Using the techniques of equivariant bifurcation theory we will study periodic solutions of the system \eqref{equ}.

It is well-known that instead of studying solutions of an arbitrary period, one can consider only $2\pi$-periodic  solutions of the parameterized system
\beq \label{eqlambda}
\ddot u(t) = -\lambda^2 \nabla U(u(t)).
\eeq
More precisely, by a standard change of variables, it can be shown that $2\pi\lambda$-periodic solutions of \eqref{equ} correspond to $2\pi$-periodic solutions of \eqref{eqlambda}. Moreover, in this case it is  obvious that we can restrict the consideration of \eqref{eqlambda} to  $\lambda\in(0,+\infty)$.

To reformulate the problem we consider a separable Hilbert space $(\h^1_{2\pi},\langle\cdot,\cdot\rangle_{\h^1_{2\pi}}),$ where
\[\h^1_{2\pi} = \{u\colon [0,2\pi] \rightarrow \bR^n\colon  u \text{ is abs. continuous, } u(0)=u(2\pi), \dot u \in L^2([0,2\pi],\bR^n)\}\]
and
\[\ds \langle u,v\rangle_{\h^1_{2\pi}} = \int_0^{2\pi} (\dot u(t), \dot v(t)) + (u(t),v(t)) \; dt.\]
With our assumptions, this space is an orthogonal representation of $\Gamma$, where the action is given by
$\Gamma \times \h^1_{2\pi} \ni (\gamma,u) \mapsto \gamma u.$
Additionally,
\begin{equation*}
\bH^1_{2\pi} = \overline{\h_0 \oplus \bigoplus_{k=1}^{\infty} \h_k},
\end{equation*}
where $\h_0=\bR^n$ and $\h_k =\{a\cos kt+b\sin kt\colon a,b\in\bR^n\}$, for $k > 0$, are orthogonal representations of $\Gamma.$

Now define a $\Gamma$-invariant functional $\Phi\colon \h^1_{2\pi}\times (0,+\infty)  \to \bR$ of class $C^2$ by the formula
\begin{equation*}  \label{functional}
\Phi(u,\lambda) = \int_0^{2\pi} \left( \frac{1}{2} \| \dot u(t) \|^2 -  \lambda^2U(u(t)) \right) \; dt.
\end{equation*}
 Notice that  $2\pi$-periodic solutions of the system
\eqref{eqlambda} can be considered as critical points of $\Phi$
and so we will study solutions of the following equation:
\beq\label{eqgrad}
\nabla_u\Phi(u,\lambda)=0.
\eeq

Fix $u_0\in(\nabla U)^{-1}(0)$ and define the constant function $\tilde{u}_0 \equiv u_0.$ Then $\tilde{u}_0$ is a stationary solution of the system \eqref{eqlambda} for every $\lambda >0.$ Since $\Gamma(u_0)\subset (\nabla U)^{-1}(0),\, \Gamma(\tilde{u}_0)$ consists of solutions of \eqref{eqlambda} for all $\lambda >0.$
In the rest of this section we use this notation, considering $\Gamma(u_0)$ as a set in $\bR^n$ and $\Gamma(\tilde{u}_0)$ in $\h^1_{2\pi}$. Consequently, we define the neighborhoods of these orbits in appropriate spaces, i.e. $\Gamma(u_0)_{\epsilon} = \bigcup_{u\in \Gamma(u_0)} B_\epsilon(\bR^n,u)\subset \bR^n$ and $\Gamma(\tilde{u}_0)_{\epsilon} = \bigcup_{u\in \Gamma(\tilde{u}_0)} B_\epsilon(\h^1_{2\pi},u)\subset \h^1_{2\pi}$, where $B_{\epsilon}(\bV, u)$ denotes the open ball in the space $\bV$, of radius $\epsilon$, centered in $u$.

Now define  $\mathcal{T}=\Gamma(\tilde{u}_0)\times(0,+\infty) \subset \h^1_{2\pi}\times(0,+\infty)$. The elements of this family are called the trivial solutions of  \eqref{eqgrad}, whereas the elements of the set
\[\mathcal{N}=\{(u,\lambda)\in  \h^1_{2\pi}\times(0,+\infty)\setminus\mathcal{T}\colon\nabla_u\Phi(u,\lambda)=0\}\]
are called the non-trivial solutions.
We will study the existence of local bifurcations of non-trivial solutions of equation \eqref{eqgrad} from $\mathcal{T}$, and so
let us first introduce the notion of a local bifurcation.

\bdf Fix $\lambda_0 > 0.$ The orbit $\Gamma(\tilde{u}_0)\times\{\lambda_0\}\subset\mathcal{T}$ is called an orbit of local bifurcation of solutions of \eqref{eqgrad} if $\Gamma(\tilde{u}_0)\times\{\lambda_0\}\subset \cl(\mathcal{N}).$

\edf

\begin{Remark} From the above definition it follows that if the orbit $\Gamma(\tilde{u}_0) \times \{\lambda_0\}$ is an orbit of local bifurcation, then there exists a sequence $(u_k, \lambda_k)$ in $\cN$ such that $\lambda_k \to \lambda_0$ and $u_k$ is a  $2\pi$-periodic solution of \eqref{eqlambda}, corresponding to $\lambda=\lambda_k$, where for all $\epsilon>0$ there is $k_0 \in \bN$ such that $u_k \in \Gamma(\tilde{u}_0)_{\epsilon}$ for $k \geq k_0$. Moreover, since the orbit $\Gamma(u_0)$ is isolated in the set of critical points of the potential $U$, for $\epsilon$ sufficiently small the obtained subsequence consists of non-stationary $2\pi$-periodic solutions.
\end{Remark}

\begin{Remark}\label{rem:phase}
To prove Theorems \ref{main-theo1} and \ref{main-theo2} we use the fact
(see Lemma 3.1 of \cite{KPR}) that the existence of a local bifurcation in the function space implies the existence of a local bifurcation in the phase space. More precisely,
the existence of a local bifurcation of solutions of \eqref{eqgrad} from the $\Gamma$-orbit $\Gamma(\tilde{u}_0)\times\{\lambda_0\}\subset\mathcal{T}$ implies the existence of a sequence $(u_k)$ of periodic solutions of the system \eqref{equ} with a sequence $(T_k)$ of (not necessarily minimal) periods such that $T_k \to 2\pi \lambda_0$ and  for any $\epsilon>0$ there exists $k_0\in\mathbb{N}$ such that $u_k([0,T_k])\subset \Gamma(u_0)_\epsilon$ for all $k\geq k_0.$
\end{Remark}

To formulate the necessary condition of a local bifurcation we use the following result given in \cite{PRS2}: the orbit $\Gamma(\tilde{u}_0)\times\{\lambda_0\}$ can be an orbit of local bifurcation only if
\begin{equation*}
\ker\nabla_u^2\Phi(\tilde{u}_0,\lambda_0)\cap\bigoplus_{k=1}^{\infty}\bH_k \neq \emptyset,
\end{equation*}
i.e. if the kernel is not fully contained in $\bH_0$.
The description of $\ker \nabla^2_u \Phi(\tilde{u}_0, \lambda_0)$ can be obtained with the use of the formulas characterizing the action of such an operator on subrepresentations $\bH_k$ given in Lemma 5.1.1 of \cite{FRR}. From this lemma it follows that
\begin{equation}\label{eq:opissigma}
\sigma(\nabla^2_u\Phi(\tilde{u}_0, \lambda))=\left\{ \frac{k^2-\lambda^2\alpha}{k^2+1} \colon \alpha \in \sigma (\nabla^2U(u_0)), k=0,1,2, \ldots\right\}.
\end{equation}
Recall that $\cB=\{\beta_1,\beta_2,\ldots,\beta_q\}$, where $\beta_1>\beta_2>\ldots>\beta_q>0$ are such that $\sigma^+(\nabla^2 U(u_0)) = \{\beta_1^2,\ldots, \beta_q^2\}$.
Put $\Lambda=\{\frac{k}{\beta}\colon k \in \bN, \beta\in\cB\}.$
Then we have:

\begin{Fact}\label{fact:necess}
If $\Gamma(\tilde{u}_0) \times \{\lambda_0\}$ is an orbit of local bifurcation of solutions of \eqref{eqgrad} then $\lambda_0 \in \Lambda.$
\end{Fact}

To obtain the sufficient condition  we use Izydorek's version of the Conley index, see  Appendix for the definition.
For $\lambda\notin\Lambda$ from Fact \ref{fact:necess} it follows that the orbit $\Gamma(\tilde{u}_0)$ is isolated in $(\nabla \Phi(\cdot, \lambda))^{-1}(0)$. Since the operator is a gradient one, this orbit is an isolated invariant set of the flow generated by $-\nabla_u \Phi(\cdot, \lambda),$ i.e. it is an isolated invariant set in the sense of the Conley index theory.
Hence the Conley index $CI_\Gamma(\Gamma(\tilde u_0), -\nabla_u \Phi(\cdot, \lambda))$ is well-defined for $\lambda\notin\Lambda$.
From the continuation property it follows that if there is a change of this index then there occurs a local bifurcation, see \cite{Izydorek}. More precisely, there holds the following theorem.

\begin{Theorem}\label{thm:loc}
Let $\lambda_0 \in \Lambda$. If $\lambda_-,\lambda_+$ are such that $[\lambda_-,\lambda_+]\cap\Lambda=\{\lambda_0\}$ and
\begin{equation}\label{eq:CI}
CI_\Gamma(\Gamma(\tilde u_0), -\nabla_u \Phi(\cdot, \lambda_-))\neq CI_\Gamma(\Gamma(\tilde u_0), -\nabla_u \Phi(\cdot, \lambda_+)),
\end{equation}
then $\Gamma(\tilde u_0) \times \{\lambda_0\}$ is an orbit of local bifurcation of solutions of \eqref{eqgrad}.
\end{Theorem}

\begin{Remark}\label{rem:nec1}
Note that from Remark \ref{rem:phase} and Fact \ref{fact:necess} it follows that the periods of solutions of \eqref{equ} in the bifurcating sequences have to satisfy $T_k \to 2 \pi \frac{k}{\beta},$ where $\frac{k}{\beta} \in \Lambda$. However, these periods do not have to be minimal. On the other hand, if we consider bifurcations from the orbits of the form $\Gamma(\tilde u_0) \times \{\frac{1}{\beta_{j_0}}\},$ where $\beta_{j_0}\in\cB$ satisfies the nonresonance condition $\beta_j \slash \beta_{j_0} \not \in \bN $ for $\beta_j\in\cB\setminus\{\beta_{j_0}\}$, then the obtained periods are minimal. Therefore to prove the assertions of Theorems \ref{main-theo1} and \ref{main-theo2} we restrict our attention to such orbits.
\end{Remark}

Now we are in a position to prove the main results of our article, namely Theorems \ref{main-theo1} and \ref{main-theo2}.

\subsection{Proof of Theorem \ref{main-theo1}}
To prove the assertion we study local bifurcations from the levels satisfying the necessary condition given in Fact \ref{fact:necess}. From Remark \ref{rem:nec1} it follows that we can restrict our attention to the set $\Lambda_0\subset \Lambda$, where
\begin{equation}\label{eq:Lambda0}
\Lambda_0=\left\{\frac{1}{\beta_{j_0}}\colon \beta_{j_0} \in\cB \text{ and }\frac{\beta_j}{\beta_{j_0}} \not \in \bN \text{ for }\beta_j\in\cB\setminus\{\beta_{j_0}\}\right\}.
\end{equation}
Obviously, if the assumptions of Theorem \ref{main-theo1} are satisfied, this set is nonempty.

Fix $\lambda_0\in\Lambda_0$. By the definition of $\Lambda$, there exists $\varepsilon>0$ such that $[\lambda_0-\varepsilon,\lambda_0+\varepsilon]\cap\Lambda=\{\lambda_0\}.$
Put $\lambda_{\pm} = \lambda_0\pm \varepsilon$. Then, since $\lambda_{\pm}\notin\Lambda$,
the Conley indices $CI_\Gamma(\Gamma(\tilde{u}_0), -\nabla_u \Phi(\cdot, \lambda_{\pm}))$ are well-defined.

By Theorem \ref{thm:loc}, to prove our assertion it is enough to show that the inequality \eqref{eq:CI} holds.
Note that $\nabla\Phi(\cdot,\lambda_\pm)$ are of the form of completely continuous perturbations of the identity and therefore completely continuous perturbations of linear, bounded, Fredholm, self adjoint operators $\nabla^2_u\Phi(\tilde u_0,\lambda_\pm)$. From the definition of the equivariant Conley index, $CI_{\Gamma}\left(\Gamma(\tilde u_0),-\nabla_u\Phi(\cdot,\lambda_\pm)\right)$ are the $\Gamma$-homotopy types of $\Gamma$-spectra of the types $(\bH_{n+1, \pm}^+)^\infty_{n=0}$, where $\bH_{n+1,\pm}^+$ are the direct sums of eigenspaces of $-\nabla^2_u\Phi(\tilde u_0,\lambda_\pm)_{|\bH_{n+1}}$ corresponding to the positive eigenvalues.
Since $\nabla^2_u\Phi(\tilde u_0,\lambda_\pm)=Id-\cL_\pm$ with $\cL_\pm$ being linear, compact, self adjoint and $\Gamma$-equivariant operators,
by the spectral theorem for compact operators it follows that the eigenvalues of $-\nabla^2_u\Phi(\tilde u_0,\lambda_\pm)$ converge to $-1$.
Consequently, for $n$ sufficiently large, $\bH_{n+1,\pm}^+=\emptyset$ and the sequences of $\Gamma$-CW-complexes $CI_\Gamma(\Gamma(\tilde{u}_0),-\nabla_u\Phi^n(\cdot,\lambda_\pm))$ stabilize.
Therefore, the condition \eqref{eq:CI} is equivalent to
\begin{equation}\label{eq:CIsk}
CI_\Gamma(\Gamma(\tilde u_0), -\nabla_u\Phi^n(\cdot,\lambda_-)) \neq CI_\Gamma(\Gamma(\tilde u_0),-\nabla_u\Phi^n(\cdot,\lambda_+))
\end{equation}
for $n$ sufficiently large. In the following we assume that such $n$ is fixed.

Consider $T_{u_0}^{\perp}\Gamma(u_0)$ - the space normal to the orbit $\Gamma(u_0)$ at $u_0$ in $\bR^n$. Identifying $\bH_0$ with $\bR^n$, we write
$\bW^n=T_{u_0}^{\perp}\Gamma(u_0)\oplus \bigoplus_{k=1}^{n} \bH_k$ for the normal space to $\Gamma(\tilde u_0)$ in $\bigoplus_{k=0}^{n} \bH_k$. Put
$\Psi^n_{\pm}=\Phi^n_{|\bW^n}(\cdot,\lambda_{\pm})\colon\bW^n\to\bR$.
It is known that the space $\bW^n$ and the functionals $\Psi^n_{\pm}$ are $H$-invariant, where $H=\Gamma_{u_0}$. Moreover, the set $\{\tilde u_0\}$ is isolated in $(\nabla\Psi^n_{\pm})^{-1}(0)$, and so $\{\tilde u_0\}$ is an isolated invariant set in the sense of the Conley index. Therefore, the Conley indices $CI_H(\{\tilde u_0\}, -\nabla\Psi^n_{\pm})$ are well-defined.
Furthermore, by Theorem \ref{thm:CIsmash} we have
\begin{equation}\label{equality1}
CI_\Gamma(\Gamma(\tilde u_0), -\nabla_u\Phi^n(\cdot,\lambda_{\pm}))=\Gamma^+ \wedge_{H}CI_H(\{\tilde u_0\}, -\nabla\Psi^n_{\pm}).
\end{equation}
Therefore to study inequality \eqref{eq:CIsk} we will first investigate the indices $CI_H(\{\tilde u_0\}, -\nabla\Psi^n_{\pm})$ and then apply the abstract results from the previous section.

Since $\tilde u_0$ is a non-degenerate critical point of  $\Psi^n_{\pm},$  we can apply Remark \ref{rem:nondegenerate}, obtaining that these indices are the homotopy types of $S^{(\bW^n)^+},$ where $(\bW^n)^+$ is the direct sum of eigenspaces of $-\nabla^2 \Psi^n_{\pm}(\tilde{u}_0)$ corresponding to the positive eigenvalues. Therefore we study the spectral decompositions of $\bW^n$ given by the isomorphisms $-\nabla^2 \Psi^n_{\pm}(\tilde{u}_0).$

Note that for both operators, $\bW^n$ can be decomposed as $$\bW^n=\bH_1 \oplus (T_{u_0}^\bot  \Gamma(u_0) \oplus \bigoplus_{k=2}^{n} \bH_k).$$ Moreover, it is easy to see that the description of the spectrum of $-\nabla^2 \Psi^n_{\pm}(\tilde{u}_0)$ is  as in formula \eqref{eq:opissigma}.
Hence, to obtain the spectral decomposition, we have to study signs of the numbers $k^2-\lambda_{\pm}^2\beta_j^2$ for $k \in \bN, \beta_j \in \cB$. Consider $k^2-\lambda^2 \beta_j^2$, for $k$ and $\beta_j$ fixed, as a continuous function of $\lambda \in [\lambda_-, \lambda_+]$ and note that from formula \eqref{eq:opissigma} it is nonzero for $\lambda \neq \frac{1}{\beta_{j_0}}$. If $k>1$ and $\beta_j \in \cB$, from the nonresonance condition it follows that the value $0$ is not attained, hence this function has a constant sign. Therefore the spectral decomposition of $\bH_k$ with $k>1$ does not depend on the choice of $\lambda_{\pm}.$
Obviously, also a decomposition of $T_{u_0}^\bot  \Gamma(u_0)$ is independent of such a choice. We denote by $\bW^-$ (respectively $\bW^{+}$) the direct sum of the eigenspaces corresponding to the negative (respectively positive) eigenvalues of
$-\nabla^2 \Psi^{n}_{\pm}(\tilde u_0)_{|\bW^n \ominus \bH_1}.$

What is left is to consider the spectral decomposition of $\bH_1$. In this case the zero value can be attained only for $\beta_j=\beta_{j_0}, \lambda=\frac{1}{\beta_{j_0}}.$ Moreover, when passing $\lambda=\frac{1}{\beta_{j_0}}$ the function $1-\lambda^2\beta_{j_0}^2$ changes its sign. Therefore
\begin{equation}\label{eq:morse}
m^-(-\nabla^2\Psi_{+}(\tilde{u}_0)_{|\bH_1}) \neq m^-(-\nabla^2\Psi_{-}(\tilde{u}_0)_{|\bH_1}),
\end{equation}
where $m^-$ denotes the Morse index.
Observe that the eigenspaces of $-\nabla^2\Psi_{\pm}(\tilde{u}_0)_{|\bH_1}$ are the same as the ones of $-\nabla^2_u\Phi(\tilde{u}_0, \lambda_{\pm})_{|\bH_1}$, so the negative and positive eigenspaces of $-\nabla^2\Psi_{\pm}(\tilde{u}_0)_{|\bH_1}$ are $\bH^-_{1,\pm}$ and $\bH^+_{1,\pm}$ respectively.
Finally, the spectral decompositions of $\bW^n$ given by the isomorphisms $-\nabla^2 \Psi^{n}_{-}(\tilde u_0)$ and $-\nabla^2 \Psi^{n}_{+}(\tilde u_0)$  are of the form
\begin{equation*}
\ds \bW^{n}=(\bH_{1,-}^-\oplus\bH_{1,-}^+)\oplus (\bW^-\oplus\bW^+),
\end{equation*}
\begin{equation*}
\ds \bW^{n}=(\bH_{1,+}^-\oplus\bH_{1,+}^+)\oplus (\bW^-\oplus\bW^+)
\end{equation*}
respectively.
Summing up,
\begin{equation} \label{equality3}
CI_H(\{ \tilde u_0\}, -\nabla\Psi^n_{\pm})=S^{\bH^+_{1,\pm}\oplus\bW^+}.
\end{equation}
Applying formulas \eqref{equality1} and \eqref{equality3} we have
\[CI_\Gamma(\Gamma(\tilde u_0), -\nabla_u\Phi^n(\cdot,\lambda_\pm))=\Gamma^+ \wedge_{H} S^{\bH^+_{1,\pm}\oplus\bW^+} ,\]
and so to finish the proof of Theorem \ref{main-theo1} it is enough to show that
\begin{equation}\label{inequality3}
\Gamma^+ \wedge_{H} S^{\bH^+_{1,-}\oplus\bW^+} \not\approx_{\Gamma} \Gamma^+ \wedge_{H} S^{\bH^+_{1,+}\oplus\bW^+}.
\end{equation}
From \eqref{eq:morse}, $\dim \bH^+_{1,-} \neq \dim\bH^+_{1,+}$. Therefore inequality \eqref{inequality3} is a consequence of Theorem \ref{thm:general}, which completes the proof of  Theorem \ref{main-theo1}.

\subsection{Proof of Theorem \ref{main-theo2}}
Throughout this proof we use the notation of the one of Theorem \ref{main-theo1}. As in this proof, to show the assertion we will study local bifurcations from the orbit $\Gamma(\tilde{u}_0) \times \{\lambda_0\}$ for a fixed $\lambda_0\in\Lambda_0$, where $\Lambda_0$ is defined by \eqref{eq:Lambda0}.
To this end we will apply an equivariant version of the so called splitting lemma, see for instance Lemma 3.2 of \cite{FRR}. We first shift the critical point $\tilde u_0$ to the origin, i.e.
we consider $H$-invariant maps $\Pi_{\pm}^n\colon \bW^n\to\bR$ given by $\Pi_{\pm}^n(u)=\Psi^n_{\pm}(u+\tilde u_0)$. Then
$$CI_H(\{ \tilde u_0\}, -\nabla\Psi^n_{\pm})=CI_H(\{0\}, -\nabla\Pi^n_{\pm}).$$

It is easy to observe that $\ker\nabla^2\Pi_{-}^n(0)=\ker\nabla^2\Pi_{+}^n(0)$ and $\im\nabla^2\Pi_{-}^n(0)=\im\nabla^2\Pi_{+}^n(0)$, for simplicity we denote these spaces respectively by $\cN$ and $\cR$.
Put $\cA_{\pm}=(\nabla^2\Pi_{\pm}^n(0))_{|\cR}$ and note that $\cA_{\pm}$ are isomorphisms.
Applying the splitting lemma we get $H$-invariant maps $\varphi_{\pm} \colon \cN \to \bR$ and $H$-invariant homotopies between $\nabla\Pi^n_{\pm}$ and $(\nabla \varphi_{\pm},\cA_{\pm})$, with $\{0\}$ being an isolated invariant set at all levels of these homotopies.
Applying the continuation property of the Conley index we obtain
$$
CI_H(\{0\},-\nabla\Pi^n_{\pm})=CI_H(\{0\},(-\nabla\varphi_{\pm},-\cA_{\pm})).
$$
Reasoning as in the proof of Lemma 3.3.1 of \cite{PRS2} one can prove that $0\in\cN$ is an isolated local maximum of $\varphi_{\pm}$.
Hence, since $\cA_{\pm}$ are isomorphisms, we can apply the product formula (see Theorem \ref{thm:CIproduct}) and obtain
\begin{equation}\label{eq:splitted}
CI_H(\{0\},(-\nabla\varphi_{\pm},-\cA_{\pm}))=CI_H(\{0\},-\nabla\varphi_{\pm})\wedge CI_H(\{0\},-\cA_{\pm}).
\end{equation}
Moreover, reasoning as in the proof of Lemma 2.2 of \cite{KPR}, we get
$$ CI_H(\{0\},-\nabla\varphi_{\pm})=S^{\cN}.$$
To find the latter factor in \eqref{eq:splitted} consider $\bV=\bigoplus_{k=2}^{n} \bH_k$ and the subspace $\bV^+$ being the direct sum of the eigenspaces of $-(\cA_{\pm})_{|\bV}$ corresponding to the positive eigenvalues. Then, as in the proof of Theorem \ref{main-theo1},
$$
CI_H(\{0\},-\cA_{\pm})=S^{\bV^+\oplus\bH^+_{1,\pm}}
$$
and $\dim \bH^+_{1,-} \neq \dim\bH^+_{1,+}$. Hence
$$CI_H(\{0\},-\nabla\Pi^n_{\pm})= S^{\cN}\wedge S^{\bV^+\oplus\bH^+_{1,\pm}}= S^{\cN\oplus \bV^+\oplus\bH^+_{1,\pm} }.
$$
Since $H\in\sub(S^1)$ and $\dim\bH^+_{1,-}\neq \dim\bH^+_{1,+}$, by Theorem \ref{thm:general} we obtain
$$
\Gamma^+ \wedge_{H} (S^{\cN\oplus\bV^+\oplus\bH^+_{1,-}}) \not\approx_{\Gamma} \Gamma^+ \wedge_{H} (S^{\cN\oplus\bV^+\oplus\bH^+_{1,+}}).
$$
Applying the equality \eqref{equality1} and Theorem \ref{thm:loc} we finish the proof.

\subsection{Final remarks and open questions}\label{sec:finalremarks}
We finish this paper with some remarks.

\begin{Remark}
Taking into consideration the assumption (1) of Theorems \ref{main-theo1} and \ref{main-theo2} the following question seems to be interesting: are the counterparts of these theorems true for $\Gamma_{u_0}$ being any closed subgroup of the Lie group $\Gamma$?
As far as we know,
this question is at present far from being solved.
\end{Remark}

\begin{Remark}
In the assertions of our theorems we obtain the existence of sequences of non-stationary periodic solutions of system \eqref{equ} emanating from the orbit. An interesting question is as follows: under the assumptions of these theorems, does it emanate a connected set of such solutions? It seems that, given the relationship between the equivariant Conley index and the degree for equivariant gradient maps, see \cite{golry}, we should be able to positively answer this question.
 
 Note that the emanation of connected sets can be obtained also with the use of the theory of degree for equivariant gradient maps (see for example \cite{BKS}, \cite{GolRyb1}). The theory of the index of an orbit defined via the degree has been developed in \cite{GolKluSte2}, \cite{GolSte}. 
\end{Remark}

\begin{Remark}\label{rem:onlygamma}
It is well-known  that $\h^1_{2\pi}$ is an orthogonal representation of the group $S^1$ with the action given by shift in time. Therefore, with the action of $\Gamma$ described above, it can also be considered as a $(\Gamma\times S^1)$-representation. However, in our paper we study bifurcations from orbits of constant solutions. In such a case the $(\Gamma\times S^1)$-orbit is the same as the $\Gamma$-orbit, i.e. $(\Gamma \times S^1)(\tilde{u}_0) \approx \Gamma(\tilde{u}_0)$.
Using our method and considering additionally the $S^1$-action one cannot obtain any additional information.
Hence, for simplicity of computation, we restrict our attention to the action of  the group~$\Gamma.$
\end{Remark}




\section{Appendix}\label{sec:appendix}

\subsection{Representations of $S^1$ and $\bZ_m$.}
In this subsection we give a description of finite-dimensional orthogonal real representations of $S^1$ and $\mathbb{Z}_m$ and recall the definitions of the two topological objects which we use to prove our abstract results.

Let $l\in\mathbb{N}$ and consider a two-dimensional $S^1$-representation (denoted by $\mathbb{R}[1,l]$) with the action of the group $S^1$ given by
$(e^{i\phi},(x,y))\mapsto \Phi(\phi)^l(x,y)^T=\Phi(l\cdot \phi)(x,y)^T,$ where
\[
\Phi(\phi)=
\left[\begin{array}{cc}
\cos\phi & -\sin\phi \\
\sin\phi & \cos\phi\end{array}\right].
\]
For $k,l\in\mathbb{N}$ we will denote by $\mathbb{R}[k,l]$ the direct sum of $k$ copies of $\mathbb{R}[1,l]$ and additionally by $\mathbb{R}[k,0]$ the trivial $k$-dimensional $S^1$-representation.
It is known that if $\mathbb{V}$ is a finite-dimensional orthogonal $S^1$-representation, then there exist finite sequences $(k_i),(m_i)$
such that $k_0\in\mathbb{N}\cup\{0\}$, $k_i,m_i\in\mathbb{N}$ and $\mathbb{V}$ is equivalent to
\begin{equation}\label{eq:repS1}
 \bR[k_0,0]\oplus\bR[k_1,m_1] \oplus \ldots \oplus \bR[k_p,m_p],
\end{equation}
see \cite{Adams}.

Analogously, when no confusion can arise, for $l\in\{1,\ldots,m-1\}$ we denote by $\bR[1,l]$ the two-dimensional irreducible representation of the finite cyclic group $\bZ_m$ with the action $(e^{i\frac{2\pi k}{m}},(x,y))\mapsto \Phi(\frac{2\pi k}{m})^l(x,y)^T$ and by $\mathbb{R}[k,0]$ the trivial $k$-dimensional $\bZ_m$-representation. As before, $\mathbb{R}[k,l]$ is the direct sum of $k$ copies of $\mathbb{R}[1,l]$. It is known that if $\bV$ is a finite-dimensional orthogonal $\bZ_m$-representation, then $\bV$ is equivalent to
\begin{equation}\label{eq:repZm}
 \bR[k_0,0]\oplus\bR[k_1,m_1] \oplus \ldots \oplus \bR[k_p,m_p],
\end{equation}
where $k_0\in\bN\cup\{0\}$, $k_1,\ldots, k_p\in\bN$, $m_1,\ldots,m_p \in \{1,\ldots,m-1\}$, see \cite{Serre}.

If now $\bV$ is an $S^1$-representation without non-trivial fixed points, i.e. such that $k_0=0$ in \eqref{eq:repS1}, then the orbit space $S(\bV)/S^1$ of the action of $S^1$ on the sphere $S(\bV)$ is called the twisted projective space. Similarly if $\bV$ is a $\bZ_m$-representation such that $k_0=0$ in \eqref{eq:repZm}, then the orbit space $S(\bV)/\bZ_m$ is called the lens complex.

Note that we use in this article results of Kawasaki from \cite{kawasaki} and therefore the above terminology is taken from this paper. However, these topological spaces are also known by different names - the twisted projective space is called the weighted projective space and the lens complex is known as the weighted lens space.

\subsection{Equivariant Conley index}
In this subsection we introduce the basic notion of the equivariant Conley index. We start with the finite-dimensional case, for a more complete exposition we refer to \cite{bartsch}, \cite{Geba}.

As before, we denote by $\Gamma$ a compact Lie group. Let $\bV$ be a real finite-dimensional orthogonal $\Gamma$-representation and $\varphi\colon\bV\to\bR$  a $\Gamma$-invariant map of class $C^2$.
Suppose that $S$ is an isolated invariant set of the flow generated by $-\nabla \varphi$. In such a situation the Conley index of $S$ is defined (see \cite{bartsch}, \cite{Geba}) as the  $\Gamma$-homotopy type of a pointed $\Gamma$-CW-complex. We denote it by $CI_\Gamma(S, -\nabla \varphi)$.

The equivariant Conley index has all the properties of the classical (non-equivariant) index given in \cite{Conley}. For the convenience of the reader we recall some of them, particularly important in the proofs of our results. We start with the product formula.

\begin{Theorem}\label{thm:CIproduct} Let $S_i$, for $i=1,2$, be isolated $\Gamma$-invariant sets for the $\Gamma$-flows generated by $-\nabla\varphi_i\colon\bV_i\to\bV_i$ which are $\Gamma$-equivariant maps of class $C^1$ and $\mathbb{V}_i$ are real finite-dimensional orthogonal $\Gamma$-representations. Then $$CI_\Gamma(S_1\times S_2, (-\nabla\varphi_1,-\nabla\varphi_2))= CI_\Gamma(S_1, -\nabla\varphi_1)\wedge CI_\Gamma(S_2, -\nabla\varphi_2).$$
\end{Theorem}

In the case of an isolated invariant set containing only a non-degenerate critical point the Conley index has a relatively simple structure. Namely, analogously as in the non-equivariant case (see \cite{SmoWas}), we have the following:

\begin{Remark}\label{rem:nondegenerate}
Suppose that $\varphi\colon \bV\to\bR$ is a $\Gamma$-invariant map of class $C^2$ and suppose that $x_0 \in\bV$ is an isolated non-degenerate critical point of $\varphi$. Then $CI_\Gamma(\{x_0\}, -\nabla \varphi)$ is the $\Gamma$-homotopy type of $S^{\bV^+}$, where $\bV^+$ is the direct sum of eigenspaces of $-\nabla^2 \varphi(x_0)$ corresponding to the positive eigenvalues.
\end{Remark}

In a more general case of the isolated invariant set containing an isolated critical orbit, there is a relation between the $\Gamma$-indices of this orbit and a critical point of the restriction, see Theorem 2.4.2 of \cite{PRS2}. The relation is given in terms of the smash product over the stabilizer of this critical point. For the convenience of the reader we recall it in the next theorem.

\begin{Theorem}\label{thm:CIsmash}
Suppose that $\varphi\colon \bV\to\bR$ is a $\Gamma$-invariant map of class $C^2$ and the orbit $\Gamma(x_0) \subset (\nabla \varphi)^{-1}(0)$ is isolated. Define $\psi=\varphi_{|T_{x_0}^{\perp}\Gamma(x_0)}$. Then
$$CI_\Gamma(\Gamma(x_0), - \nabla \varphi)=\Gamma^+ \wedge_{\Gamma_{x_0}} CI_{\Gamma_{x_0}}(\{x_0\},-\nabla \psi).$$
\end{Theorem}

Let us now consider the infinite-dimensional case. First we recall the notion of a $\Gamma$-spectrum.
Let $\xi=(\mathbb{V}_n)^\infty_{n=0}$ be a sequence of finite-dimensional orthogonal representations of the group $\Gamma.$
The pair $\mathcal{E}=((\mathcal{E}_n)_{n=n(\mathcal{E})}^{\infty},(\epsilon_n)_{n=n(\mathcal{E})}^{\infty})$ is called a $\Gamma$-spectrum of the type $\xi$ if, for every $n\geq n(\mathcal{E})$,
$\mathcal{E}_{n}$ is a finite pointed $\Gamma$-CW-complex, $\epsilon_n\colon S^{\mathbb{V}_n}\wedge \mathcal{E}_n\rightarrow \mathcal{E}_{n+1}$ is a morphism  and there exists $n_0>n(\mathcal{E})$ such that $\epsilon_n$ is a $\Gamma$-homotopy equivalence for $n\geq n_0.$

The equivariant Conley index in the infinite-dimensional situation has been defined by Izydorek, see \cite{Izydorek}, as the $\Gamma$-homotopy type of a $\Gamma$-spectrum.
Izydorek's definition is given for a general $\cL\cS$-flow. Since in our paper we do not need the equivariant Conley index in the general case, we will briefly sketch the definition in the simplified situation, appropriate in our applications.

Consider an infinite-dimensional separable Hilbert space $\bH$ which is an orthogonal representation of $\Gamma$.
Let $\Phi\colon\bH\rightarrow\bR$ be a functional of the form $\Phi(u)=\frac12\langle Lu,u\rangle - K(u)$ such that $L$ is a linear, bounded, self adjoint, Fredholm and $\Gamma$-equivariant operator and $\nabla K$ is a $\Gamma$-equivariant completely continuous operator of class $C^1$.

Suppose that $\bH = \overline{ \bigoplus_{k=0}^{\infty} \h_k},$ where $\h_k$ are disjoint orthogonal finite-dimensional representations of $\Gamma$ such that $\bH_0=\ker L$ and $L(\bH_k)=\bH_k$ for every $k$.
Suppose that $S$ is an isolated invariant set of the flow generated by $-\nabla\Phi$ and put  $\bH^n=\bigoplus_{k=0}^{n} \h_k$.
If $\mathcal{O}$ is an isolating $\Gamma$-neighborhood of $S$, then $\mathcal{O}\cap \bH^n$ is an isolating $\Gamma$-neighborhood for the flow generated by $-\nabla\Phi_{|\bH^{n}}$ for $n$ sufficiently large.
Denote by $S_n$ the maximal invariant subset in $\mathcal{O}\cap \bH^n$ and consider $\mathcal{E}_n=CI_{\Gamma}(S_n,-\nabla\Phi_{|\bH^{n}})$.
Then the Conley index of $S$, denoted as in the finite-dimensional case by  $CI_\Gamma(S, -\nabla\Phi)$, is defined as the $\Gamma$-homotopy type of a $\Gamma$-spectrum $(\mathcal{E}_n,\epsilon_n)$ of the type $(\bH_{n+1}^+)^\infty_{n=0}$, where $\bH_{n+1}^+$ is the direct sum of eigenspaces of $-L_{|\bH_{n+1}}$ corresponding to the positive eigenvalues and $\epsilon_n\colon S^{\bH_{n+1}^+}\wedge \mathcal{E}_n\rightarrow \mathcal{E}_{n+1}$ are some morphisms obtained via the product formula given in Theorem \ref{thm:CIproduct}.


\begin{thebibliography}{99}
\bibitem{Adams} J. F. Adams, \textit{Lectures on Lie groups}, W. A. Benjamin Inc., New York-Amsterdam , 1969.
\bibitem{BKS}  Z. Balanov, W. Krawcewicz, H. Steinlein, \textit{Applied Equivariant Degree}, AIMS Series on Differential Equations \&  Dynamical Systems, Vol. 1, Springfield, 2006.
\bibitem{bartsch} T. Bartsch, \emph{Topological methods for variational problems with symmetries}, \rm Lecture Notes in Mathematics 1560, Springer-Verlag, Berlin, 1993.
\bibitem{bartsch1} T. Bartsch, {\em A generalization of the Weinstein-Moser theorems on periodic orbits of a Hamiltonian system near an equilibrium},  Ann. Inst. H. Poincar\'e Anal. Non Lin\'eaire 14(6) (1997), 691--718.
\bibitem{danryb} E. N. Dancer, S. Rybicki, {\em  A note on periodic solutions  of autonomous Hamiltonian systems
emanating from degenerate stationary solutions}, Differential Integral Equations 12(2) (1999).
\bibitem{dieck} T. tom Dieck, \emph{Transformation groups}, Walter de Gruyter \& Co., Berlin, 1987.
\bibitem{Conley} C. Conley, \textit{Isolated invariants sets and the Morse index}, CBMS Regional Conference Series in Mathematics 38, American Mathematical Society, Providence, R. I., 1978.
\bibitem{fadrab} E. Fadell, P.H. Rabinowitz, \textit{Generalized cohomological index theories for Lie group actions with an application to bifurcation questions for Hamiltonian systems},  Invent. Math. 45(2) (1978), 139--174.
\bibitem{FRR} J. Fura, A. Ratajczak, S. Rybicki, \emph{Existence and continuation of periodic solutions of autonomous Newtonian systems},  J. Differential Equations 218(1) (2005), 216--252.
\bibitem{Geba} K. G\c{e}ba, \emph{Degree for gradient equivariant maps and equivariant Conley index}, Topological nonlinear analysis II, Birkhäuser (1997),  247--272.
\bibitem{Geba1} K. G\c{e}ba, Private communication.
\bibitem{GolKluSte2} A. Go{\l}\c{e}biewska, J. Kluczenko, P. Stefaniak, \emph{Bifurcations from degenerate orbits of solutions of nonlinear elliptic systems}, arXiv:2107.00408.
\bibitem{GolRyb1} A. Go{\l}\c{e}biewska, S. Rybicki, \emph{Global bifurcations of critical orbits of $G$-invariant strongly indefinite functionals}, Nonlinear Anal. 74(5) \rm (2011), 1823--1834.
\bibitem{golry}  A. Go\l\c{e}biewska, S.  Rybicki, {\em Equivariant Conley index versus  degree for equivariant gradient maps}, Discrete Contin. Dyn. Syst. Ser. S 6(4) (2013), 985--997.
\bibitem{GolSte}  A. Go{\l}\c{e}biewska, P. Stefaniak, \emph{Global bifurcation from an orbit of solutions to non-cooperative semi-linear Neumann problem}, J. Differential Equations 268(11) (2020) 6702--6728.
\bibitem{hatcher} A. Hatcher, \textit{Algebraic topology}, Cambridge University Press, Cambridge, 2002.
\bibitem{Illman} S. Illman, \emph{The equivariant triangulation theorem for actions of compact Lie groups}, Math. Ann. 262(4) (1983), 487--501.
\bibitem{Izydorek} M. Izydorek, \emph{Equivariant Conley index in Hilbert spaces and applications to strongly indefinite problems}, Nonlinear Anal.  51(1) (2002), 33--66.
\bibitem{Kawakubo} K. Kawakubo, \textit{The theory of transformation groups}, Oxford University Press, New York, 1991.
\bibitem{kawasaki} T. Kawasaki, \textit{Cohomology of twisted projective spaces and lens complexes}, Math. Ann. 206 (1973), 243--248.
\bibitem{KPR} M. Kowalczyk, E. P\'{e}rez-Chavela, S. Rybicki, \emph{Symmetric Lyapunov center theorem for orbit with nontrivial isotropy group},  Adv. Differential Equations 25(1-2) (2020), 1--30.
\bibitem{lyapunov} A. M. Lyapunov, \textit{Probl\`eme g\'en\'eral de la stabili\'te du mouvement,} Ann. Fac. Sci. Univ. Toulouse (2) 9 (1907), 203--474.
\bibitem{mayer} K. Mayer, \textit{G-invariante Morse-funktionen}, Manuscripta Math. 63(1) (1989), 99--114.
\bibitem{mawhin} J. Mawhin, M. Willem, \textit{Critical Point Theory and Hamiltonian Systems}, Springer-Verlag, New York, 1989.
\bibitem{morost} J. A. Montaldi, R. M. Roberts, I. N. Stewart, \emph{Periodic solutions near equilibria of symmetric Hamiltonian systems}, Phil. Trans. R. Soc. Lond. A 325(1584) (1988), 237--293.
\bibitem{moser} J. Moser, \textit{Periodic orbits near an equilibrium and a theorem by Alan Weinstein}, Comm. Pure Appl. Math. 29(6) (1976), 724--747.
\bibitem{PRS} E. P\'{e}rez-Chavela, S. Rybicki, D. Strzelecki, \emph{Symmetric Liapunov center theorem}, Calc. Var. Partial Differential Equations 56(2) (2017), doi:10.1007/s00526-017-1120-1.
\bibitem{PRS2}  E. P\'{e}rez-Chavela, S. Rybicki, D. Strzelecki, \emph{Symmetric Liapunov center theorem for minimal orbit}, J. Differential Equations 265(3) (2018), 752--778.
\bibitem{Serre} J. P. Serre,  \textit{Linear representations of finite groups}, Graduate Texts in Mathematics 42, Springer-Verlag, New York-Heidelberg, 1977.
\bibitem{SmoWas} J. Smoller, A. Wasserman, \emph{Bifurcation and symmetry-breaking}, Invent. Math. 100(1) (1990), 63--95.
\bibitem{Strzelecki} D. Strzelecki, \emph{Periodic solutions of symmetric Hamiltonian systems}, Arch. Ration. Mech. Anal. 237(2) (2020), 921--950.
\bibitem{szulkin} A. Szulkin, {\em Bifurcation for strongly indefinite functionals and a Liapunov type theorem for Hamiltonian systems},  Differential Integral Equations 7(1) (1994), 217--234.
\bibitem{weinstein} A. Weinstein, \textit{Normal modes for nonlinear Hamiltonian systems}, Invent. Math. 20 (1973), 47--57.
\end{thebibliography}
\end{document}